\documentclass{amsproc}

\usepackage{amssymb}
\usepackage[numbers,sort&compress]{natbib}
\usepackage{amscd}
\usepackage{tikz}

\def\N{\mathbb{N}}

\def\C{\mathbb{C}}
\def\No{\mathbb{N}_0}

\def\ri{\mathrm{i}}

\newcommand\abs[1]{\lvert#1\rvert}

\newcommand\norm[1]{\lVert #1\rVert}
\def\dm{\mathrm d}

\def\emptynorm{\lVert\cdot\rVert}
\def\emptyabs{\lvert\cdot\rvert}

\def\adj{^*}
\def\pin{^-}
\def\pfi{^+}

\newcommand{\pl}[1]{\Lambda(#1)}
\newcommand\fdiff[1]{^{[#1]}}
\newcommand\diff[1]{^{(#1)}}
\newcommand\pgen[1]{\mathfrak{p}(#1)}
\newcommand\mgen[1]{\mathfrak{m}(#1)}

\def\F{\mathcal{F}}
\def\G{\mathcal{G}}

\def\pln{^{\rm no}}
\def\plp{^{\rm pl}}

\newcommand\od[2][{1}]{D^{\ifx\relax#1\relax\else(#1)\fi}(#2)}
\newcommand\fd[3][{\F}]{D_{#1}^{(#2)}(#3)}
\newcommand\ddf[2][{\F}]{D_{#1}(#2)}

\newtheorem{theorem}{Theorem}[section]

\newtheorem{proposition}[theorem]{Proposition}

\newtheorem{corollary}[theorem]{Corollary}
\newtheorem{lemma}[theorem]{Lemma}
\newtheorem{question}[theorem]{Question}

\theoremstyle{definition}
\newtheorem{definition}[theorem]{Definition}

\theoremstyle{remark}

\newtheorem{example}[theorem]{Example}

\begin{document}
\title[The chain rule for $\F$-differentiation]{The chain rule for $\F$-differentiation}

\author{T. Chaobankoh}
\address{Department of Mathematics, Faculty of Science, Chiang Mai University, Chiang Mai, 50200, Thailand}
\email{tanadon.c@cmu.ac.th}
\thanks{The first author was supported by The Higher Educational Strategic Scholarships for Frontier Research Network, Thailand}

\author{J. F. Feinstein}
\address{School of Mathematical Sciences, The University of Nottingham, University Park, Nottingham, NG7 2RD, UK}
\email{joel.feinstein@nottingham.ac.uk}

\author{S. Morley}
\address{School of Mathematical Sciences, The University of Nottingham, University Park, Nottingham, NG7 2RD, UK}
\email{pmxsm9@nottingham.ac.uk}
\thanks{The third author was supported by a grant from the EPSRC}
\thanks{This paper contains work from the PhD theses of the first and third authors}

\subjclass[2010]{Primary 46J10, 46J15; Secondary 46E25}
\keywords{Differentiable functions, Banach function algebra, completions, Normed algebra}
\begin{abstract}
Let $X$ be a perfect, compact subset of the complex plane, and let $D^{(1)}(X)$ denote the (complex) algebra of continuously complex-differentiable functions on $X$. Then $D^{(1)}(X)$ is a normed algebra of functions but, in some cases, fails to be a Banach function algebra. Bland and the second author (\cite{bland2005}) investigated the completion of the algebra $D^{(1)}(X)$, for certain sets $X$ and collections $\mathcal{F}$ of paths in $X$, by considering $\mathcal{F}$-differentiable functions on $X$.

In this paper, we investigate composition, the chain rule, and the quotient rule for this notion of differentiability. We give an example where the chain rule fails, and give a number of sufficient conditions for the chain rule to hold. Where the chain rule holds, we observe that the Fa\'a di Bruno formula for higher derivatives is valid, and this allows us to give some results on homomorphisms between certain algebras of $\mathcal{F}$-differentiable functions.
\end{abstract}
\maketitle

Throughout this paper, we use the term {\em compact plane set} to mean a non-empty, compact subset of the complex plane, $\C$. We denote the set of all positive integers by $\N$ and the set of all non-negative integers by $\No$. Let $X$ be a compact Hausdorff space. We denote the algebra of all continuous, complex-valued functions on $X$ by $C(X)$ and we give $C(X)$ the uniform norm $\emptyabs_X$, defined by
\[
\abs{f}_X=\sup_{x\in X}\abs{f(x)}\qquad(f\in C(X)).
\]
This makes $C(X)$ into a commutative, unital Banach algebra. A subset $S$ of $C(X)$ {\em separates the points of $X$} if, for each $x,y\in X$ with $x\neq y$, there exists $f\in S$ such that $f(x)\neq f(y)$. A {\em normed function algebra} on $X$ a normed algebra $(A,\emptynorm)$ such that $A$ is a subalgebra of $C(X)$, $A$ contains all constant functions and separates the points of $X$, and, for each $f\in A$, $\norm{f}\geq\abs{f}_X$. A {\em Banach function algebra} on $X$ is a normed function algebra on $X$ which is complete. We say that such a Banach function algebra $A$ is {\em natural} (on $X$) if every character on $A$ is given by evaluation at some point of $X$. We refer the reader to \cite{dales2000} (especially Chapter~4) for further information on Banach algebras and Banach function algebras.

Let $\od[1]X$ denote the normed algebra of all continuously (complex) differentiable, complex-valued functions on $X$, as discussed in \cite{dales1973} and \cite{dalefein2010}. Furthermore, let $\od[n]X$ denote the normed algebra of all continuously $n$-times (complex) differentiable, complex-valued functions on $X$, and let $\od[\infty]X$ denote the algebra of continuous functions which have continuous (complex) derivatives of all orders. Dales and Davie (\cite{dales1973}) also introduced the algebras
\[
\od[]{X,M}:=\left\{f\in\od[\infty]{X}:\sum_{j=0}^\infty\frac{\abs{f\diff j}_X}{M_j}<\infty\right\},
\]
where $M=(M_n)_{n=0}^\infty$ is a suitable sequence of positive real numbers. The algebras $\od[]{X,M}$ are called {\em Dales-Davie} algebras.

The usual norms on the algebras $\od[n]{X}$ ($n\in\N$) and $\od[]{X,M}$ above need not be complete, so we often investigate the completion of these algebras. One approach to this was introduced by Bland and Feinstein \cite{bland2005}, where they discussed algebras of $\F$-differentiable functions (see Section \ref{sec Bland-Feinstein algebras}), and these algebras were investigated further in \cite{dalefein2010} and \cite{hoffmann2011}.

Kamowitz and Feinstein investigated the conditions under which composition with an infinitely differentiable map induces an endomorphism (\cite{kamowitz1998,feinstein2000a,feinstein2004a}) or a homomorphism (\cite{feinstein2004k}) between Dales-Davie algebras.

In this paper, we investigate composition, the chain rule, and the quotient rule for $\F$-differentiation. We give an example where the chain rule for $\F$-differentiation fails, and give a number of sufficient conditions for the chain rule to hold. We also prove a version of the quotient rule for $\F$-differentiable functions.

Where the chain rule holds, we observe that the Fa\'a di Bruno formula for higher derivatives is valid, and this allows us to give some sufficient conditions, similar to those in \cite{feinstein2004k}, for composition with an infinitely $\F$-differentiable function to induce a homomorphism between the $\F$-differentiability versions of Dales-Davie algebras.

\section{Paths in the complex plane}
\label{paths section}
We begin with a discussion of collections of paths in the complex plane.

\begin{definition}
A {\em path} in $\C$ is a continuous function $\gamma:[a,b]\to\C,$ where $a<b$ are real numbers. Let $\gamma:[a,b]\to\C$ be a path. The {\em parameter interval of $\gamma$} is the interval $[a,b]$. The {\em endpoints of $\gamma$} are the points $\gamma(a)$ and $\gamma(b),$ which we denote by $\gamma\pin$ and $\gamma\pfi,$ respectively. We denote by $\gamma\adj$ the {\em image} $\gamma([a,b])$ of $\gamma$. A {\em subpath} of $\gamma$ is a path obtained by restricting $\gamma$ to a non-degenerate$,$ closed subinterval of $[a,b]$. If $X$ is a subset of $\C$ then we say that {\em $\gamma$ is a path in $X$} if $\gamma\adj\subseteq X$.
\end{definition}

Let $\gamma$ be a path in $\C$. We say that $\gamma$ is a {\em Jordan path} if $\gamma$ is an injective function.

Let $[a,b]$ be a non-degenerate closed interval. A {\em partition} $\mathcal P$ of $[a,b]$ is a finite set $\{x_0,\dotsc,x_n\}\subseteq[a,b]$ such that $x_0=a$, $x_n=b$ and $x_j<x_{j+1}$ for each $j\in\{0,1,\dotsc,n-1\}$. If $\mathcal P$ and $\mathcal P'$ are partitions of $[a,b]$ then we say that $\mathcal P'$ is {\em finer} than $\mathcal P$ if $\mathcal P\subseteq\mathcal P'$.

\begin{definition}
Let $\gamma:[a,b]\to\C$ be a path and let $c,d\in[a,b]$ with $c<d$. The {\em total variation} of $\gamma$ over $[c,d]$ is
\[
V_c^d(\gamma):=\sup\left\{\sum_{j=0}^{n-1}\abs{\gamma(x_{j+1})-\gamma(x_j)}:\mathcal P=\{x_0,\dotsc,x_n\}\right\}
\]
where the supremum is taken over all partitions $\mathcal P$ of $[c,d]$. We say that $\gamma$ is {\em rectifiable} if $V_a^b(\gamma)<\infty,$ in which case we write $\pl\gamma:=V_a^b(\gamma);$ otherwise it is {\em non-rectifiable}. The {\em length} of a rectifiable path $\gamma$ is $\pl\gamma$.
\end{definition}

For a detailed discussion of paths, total variation and path length, see \cite[Chapter~6]{apostol1974}.

We say that a path $\gamma$ is {\em admissible} if $\gamma$ is rectifiable and contains no constant subpaths. Let $\gamma:[a,b]\to\C$ be a non-constant (but not necessarily admissible) rectifiable path. We define the {\em path length parametrisation} $\gamma\plp:[0,\pl\gamma]\to\C$ of $\gamma$ to be the unique path satisfying $\gamma\plp(V_a^t(\gamma))=\gamma(t)$ ($t\in [a,b]$); see, for example, \cite[pp.~109-110]{federer1969} for details. We define the {\em normalised path length parametrisation} $\gamma\pln:[0,1]\to\C$ of $\gamma$ to be the path such that $\gamma\pln(t)=\gamma\plp(t\pl\gamma)$ for each $t\in[0,1]$. It is clear that $\gamma\plp$ and $\gamma\pln$ are necessarily admissible paths and $(\gamma\plp)\adj=(\gamma\pln)\adj=\gamma\adj$. It is not hard to show, using \cite[Theorem~2.4.18]{federer1969}, that
\[
\int_\gamma f(z)\;\dm z=\int_{\gamma\plp}f(z)\;\dm z=\int_{\gamma\pln} f(z)\;\dm z,
\]
for all $f\in C(\gamma\adj)$. We shall use this fact implicitly throughout.

\begin{definition}
Let $X$ be a compact plane set and let $\F$ be a collection of paths in $X$. We define $\F\adj:=\{\gamma\adj:\gamma\in\F\}$. We say that $\F$ is {\em effective} if $\overline{\bigcup\F\adj}=X,$ each path in $\F$ is admissible$,$ and every subpath of a path in $\F$ belongs to $\F$. We denote by $\F\pln$ the collection $\{\gamma\pln:\gamma\in\F\}$.
\end{definition}

Let $X$ be a compact plane set and let $\F$ be a collection of paths in $X$. It is clear that $\F\adj=(\F\pln)\adj$.

We introduce the following definitions from \cite{bland2005,dales1973} and \cite{dalefein2010}.

\begin{definition}
Let $X$ be a compact plane set. We say that $X$ is {\em uniformly regular} if there exists a constant $C>0$ such that$,$ for all $x,y\in X,$ there exists a rectifiable path $\gamma$ in $X$ with $\gamma\pin=x$ and $\gamma\pfi=y$ such that $\pl\gamma\leq C\abs{x-y}$. We say that $X$ is {\em pointwise regular}$,$ if for each $x\in X,$ there exists a constant $C_x>0$ such that$,$ for each $y\in X,$ there exists a path $\gamma$ in $X$ with $\gamma\pin=x$ and $\gamma\pfi=y$ such that $\pl\gamma\leq C_x\abs{x-y}$.
We say that $X$ is {\em semi-rectifiable} if the union of the images of all rectifiable, Jordan paths in $X$ is dense in $X$.
\end{definition}

We also require the following definition from \cite{dalefein2010}.

\begin{definition}
Let $X$ be a compact plane set and let $\F$ be an effective collection of paths in $X$. We say that $X$ is {\em $\F$-regular at $x\in X$} if there exists a constant $C_x>0$ such that for each $y\in X$ there exists $\gamma\in\F$ with $\gamma\pin=x,$ $\gamma\pfi=y$ and $\pl\gamma\leq C_x\abs{x-y}$. We say that $X$ is {\em $\F$-regular} if $X$ is $\F$-regular at each point $x\in X$.
\end{definition}

We note that if $X$ is a compact plane set which is $\F$-regular, for some effective collection $\F$ of paths in $X$, then $X$ is pointwise regular.

\section{Algebras of $\F$-differentiable functions}
\label{sec Bland-Feinstein algebras}
In this section we discuss algebras of $\F$-differentiable functions as investigated in \cite{bland2005} and \cite{dalefein2010}, along with algebras of $\F$-differentiable functions analogous to the Dales-Davie algebras introduced in \cite{dales1973}.

\begin{definition}
Let $X$ be a perfect compact plane set$,$ let $\F$ be a collection of rectifiable paths in $X,$ and let $f\in C(X)$. A function $g\in C(X)$ is an {\em $\F$-derivative of $f$} if$,$ for each $\gamma\in\F,$ we have
\[
\int_{\gamma}g(z)\;\dm z=f(\gamma\pfi)-f(\gamma\pin).
\]
If $f$ has an $\F$-derivative on $X$ then we say that $f$ is $\F$-differentiable on $X$.
\end{definition}

The following proposition summarises several properties of $\F$-derivatives and $\F$-differentiable functions on certain compact plane sets. Details can be found in \cite{bland2005} and \cite{dalefein2010}.

\begin{proposition}\label{F derive old properties}
Let $X$ be a semi-rectifiable compact plane set and let $\F$ be an effective collection of paths in $X$.
\begin{enumerate}
  \item Let $f,g,h\in C(X)$ be such that $g$ and $h$ are $\F$-derivatives for $f$. Then $g=h$.
  \item Let $f\in\od X$. Then the usual complex derivative of $f$ on $X,$ $f',$ is an $\F$-derivative for $f$.
  \item Let $f_1,f_2,g_1,g_2\in C(X)$ be such that $g_1$ is an $\F$-derivative for $f_1$ and $g_2$ is an $\F$-derivative for $f_2$. Then $f_1g_2+g_1f_2$ is an $\F$-derivative for $f_1f_2$.
  \item Let $f_1,f_2,g_1,g_2\in C(X)$ and $\alpha,\beta\in\C$ be such that $g_1$ is an $\F$-derivative for $f_1$ and $g_2$ is an $\F$-derivative for $f_2$. Then $\alpha g_1+\beta g_2$ is an $\F$-derivative for $\alpha f_1+\beta f_2$.
\end{enumerate}
\end{proposition}

Let $X$ be a semi-rectifiable compact plane set, and let $\F$ be an effective collection of paths in $X$. In this setting we write $f\fdiff 1$ for the unique $\F$-derivative of an $\F$-differentiable function and we will often write $f\fdiff 0$ for $f$. We write $\fd1X$ for the algebra of all $\F$-differentiable functions on $X$. We note that, with the norm $\norm{f}_{\F,1}:=\abs{f}_X+\abs{f\fdiff 1}_{X}$ ($f\in\fd1X$), the algebra $\fd1X$ is a Banach function algebra on $X$ (\cite[Theorem~5.6]{dalefein2010}).

For each $n\in\N$, we define (inductively) the algebra
\[
\fd nX:=\{f\in\fd{1}X:f\fdiff 1\in\fd{n-1}X\},
\]
and, for each $f\in\fd nX,$ we write $f\fdiff n$ for the $n$th $\F$-derivative of $f$. We note that, for each $n\in\N$, $\fd nX$ is a Banach function algebra on $X$ (see \cite{bland2005}) when given the norm
\[
\norm{f}_{\F,n}:=\sum_{k=0}^n\frac{\abs{f\fdiff k}_X}{k!}\qquad(f\in\fd nX).
\]
In addition, we define the algebra $\fd\infty X$ of all functions which have $\F$-derivatives of all orders; that is, $\fd\infty X=\bigcap_{n=1}^\infty\fd nX$. It is easy to see that, for each $n\in\N$, we have $\od[n]X\subseteq\fd nX$ and $\od[\infty]X\subseteq\fd\infty X$.

\section{Maximal collections and compatibility}
\label{maximal compatibility section}
We aim to prove a chain rule for $\F$-differentiable functions, but first we must investigate collections of paths further. Throughout this section, let $X$ be a semi-rectifiable, compact plane set, and let $\mathcal A$ be the collection of all admissible paths in $X$. In this section, we identify $\fd 1X$ with the subset $S_{\F}$ of $C(X)\times C(X)$ consisting of all pairs $(f,g)$ where $f\in \fd 1X$ and $g$ is the $\F$-derivative of $f$. We begin with a definition.

\begin{definition}
Let $\gamma$ be an admissible path in $X$ and let $f,g\in C(X)$. We say that $g$ is the {\em $\gamma$-derivative} of $f$ if$,$ for each subpath $\sigma$ of $\gamma$, we have
\[
\int_\sigma g(z)\;\dm z=f(\sigma\pfi)-f(\sigma\pin).
\]
\end{definition}

Note that, in the above, if $\G$ denotes the collection of all subpaths of $\gamma$, then $\G$ is effective in $\gamma\adj$, so $\G$-derivatives on $\gamma\adj$ are unique. Thus, if $f$ has a $\gamma$ derivative $g$ on $X$, then $g|_{\gamma\adj}$ is uniquely determined.

\begin{definition}
Let $S\subseteq C(X)\times C(X)$. We define
\[
\pgen{S}:=\{\gamma\in\mathcal A:\text{for all $(f,g)\in S,$ $g$ is the $\gamma$-derivative of $f$}\}.
\]
Let $\F$ be an effective collection of paths in $X$. Then we write $\mgen{\F}$ for $\pgen{S_{\F}},$ where $S_{\F}=\{(f,f\fdiff 1):f\in\fd 1X\}$ as above.
\end{definition}

The following lemma follows quickly from the definition of $\mgen{\F}$. We omit the details.

\begin{lemma}\label{maximal collection same algebra}
Let $S,T\subseteq C(X)\times C(X)$. If $S\subseteq T$ then $\pgen T\subseteq\pgen S$. Let $\F,\G$ be effective collections of paths in $X$. Then $\fd[\mgen\F]1X=\fd1X$. Moreover$,$ $\fd1X\subseteq\fd[\G]1X$ if and only if $\mgen\G\subseteq\mgen\F$.
\end{lemma}

We now investigate some operations on $\mgen{\F}$. Let $\gamma_1:[a,b]\to\C$ and $\gamma_2:[c,d]\to\C$ be paths such that $\gamma_1\pfi=\gamma_2\pin$. We write $\gamma_1\dotplus\gamma_2$ for the path given by
\[
(\gamma_1\dotplus\gamma_2)(t)=
\begin{cases}
\gamma_1(a+2t(b-a)),& t\in[0,1/2),\\
\gamma_2(c+(2t-1)(d-c)),& t\in[1/2,1].
\end{cases}
\]
We call the path $\gamma_1\dotplus\gamma_2$ the {\em join} of $\gamma_1$ and $\gamma_2$.
Note that, if $\gamma_1$ and $\gamma_2$ are admissible, then $\gamma_1\dotplus\gamma_2$ is admissible. The {\em reverse of $\gamma_1$}, denoted by $-\gamma_1$, is given by $-\gamma_1(t)=\gamma_1(b-t(b-a))$ ($t\in[0,1]$). Our notation for joining and reversing is not entirely standard and there are many ways to parametrise these paths.

\begin{lemma}
Let $\F$ be an effective collection of paths in $X$. Then $\mgen{\F}$ has the following properties$:$
\begin{enumerate}
 \item if $\gamma\in\mgen{\F}$ then $-\gamma\in\mgen{\F};$
 \item if $\gamma\in\mgen{\F}$ then $\gamma\plp,\gamma\pln\in\mgen\F;$
 \item if $\gamma_1,\gamma_2\in\mgen{\F}$ such that $\gamma_1\pfi=\gamma_2\pin$ then $\gamma_1\dotplus\gamma_2\in\mgen{\F}$.
\end{enumerate}
\end{lemma}
\begin{proof}
(a) This is clear from the definitions.

(b) This is clear from the definitions, and the discussions in Section \ref{paths section}.

(c) This is effectively \cite[Theorem~4.5]{bland2005}, and follows from the definitions.
\end{proof}

We also make the following observation about collections of paths generated by a set in $C(X)\times C(X)$ and its closure in the norm given by $\norm{(f,g)}_1=\abs f_X+\abs g_X$ for each $(f,g)\in C(X)\times C(X)$.

\begin{lemma}\label{path generated dense set}
Let $S\subseteq C(X)\times C(X)$. Then $\pgen{\overline S}=\pgen{S}$ where the closure of $S$ is taken in the norm $\emptynorm_1$ on $C(X)\times C(X)$ as above.
\end{lemma}
\begin{proof}
By Lemma \ref{maximal collection same algebra}, we have $\pgen{\overline S}\subseteq\pgen S$. Let $(f_n,g_n)$ be a sequence in $S$ such that $(f_n,g_n)\to (f,g)\in \overline S$ as $n\to\infty$. Let $\gamma\in\pgen S$. Then we have
\[
f_n(\gamma\pfi)-f_n(\gamma\pin)=\int_\gamma g_n(z)\;\dm z,
\]
for all $n\in\N$. But now $f_n\to f$ and $g_n\to g$ uniformly as $n\to\infty$, so
\[
f(\gamma\pfi)-f(\gamma\pin)=\lim_{n\to\infty}f_n(\gamma\pfi)-f_n(\gamma\pin) =\lim_{n\to\infty}\int_\gamma g_n(z)\;\dm z=\int_{\gamma}g(z)\;\dm z.
\]
Thus $\gamma\in\pgen{\overline S}$. It follows that $\pgen{S}\subseteq\pgen{\overline S}$ and this completes the proof.
\end{proof}

We also require the following elementary lemma, which is a minor variant of a standard result. We include a proof for the convenience of the reader.

\begin{lemma}\label{composition rectifiable}
Let $\F$ be an effective collection of paths in $X$. Let $\varphi\in\fd1X$ and let $\gamma:[a,b]\to X\in\F$. Then $\varphi\circ\gamma$ is rectifiable$,$ and hence$,$ if $\varphi\circ\gamma$ is non-constant$,$ $(\varphi\circ\gamma)\plp$ is admissible.
\end{lemma}
\begin{proof}
Let $\mathcal P=\{x_0,\dotsc,x_n\}$ be a partition of $[a,b]$. Set $\sigma:=\varphi\circ\gamma$ and let $\gamma_j$ be the subpath of $\gamma$ obtained by restricting $\gamma$ to $[x_j,x_{j+1}]$ for each $j\in\{0,\dotsc,n-1\}$. We have
\[
V(\sigma,\mathcal P):=\sum_{j=0}^{n-1}\abs{\sigma(x_{j+1})-\sigma(x_j)} =\sum_{j=0}^{n-1}\left\lvert\int_{\gamma_j}\varphi\fdiff 1(z)\;\dm z\right\rvert\leq \abs{\varphi\fdiff 1}_X\pl\gamma.
\]
It follows that $\pl\sigma=\sup V(\sigma,\mathcal P)\leq \abs{\varphi\fdiff 1}_X\pl\gamma$, where the supremum is taken over all partitions $\mathcal P$ of $[a,b]$. As noted earlier, if $\sigma$ is non-constant then $\sigma\plp$ is admissible. This completes the proof.
\end{proof}

We now introduce our notion of compatibility.

\begin{definition}\label{compatibility definition}
Let $Y$ be a semi-rectifiable compact plane set$,$ let $\F$ be an effective collection of paths in $X$ and let $\G$ be an effective collection of paths in $Y$. Let $\varphi\in\fd1X$ such that $\varphi(X)\subseteq Y$. We say that $\varphi$ is {\em $\F$-$\G$-compatible} if$,$ for each $\gamma\in\F$, either $\varphi\circ\gamma$ is constant or we have $(\varphi\circ\gamma)\plp\in\mgen{\G}$.
\end{definition}

Let $X,Y$ be semi-rectifiable compact plane sets, let $\F$ be an effective collection of paths in $X$, and $\G$ be an effective collection of paths in $Y$. If $\mgen\G$ is the collection of all admissible paths in $Y$ then, for any $\varphi\in\fd 1X$ with $\varphi(X)\subseteq Y,$ $\varphi$ is $\F$-$\G$-compatible.

\begin{example}\label{incompatibility example}
Let $X:=\{x+\ri y\in\C:x,y\in[0,1]\}$. Let $\F$ be the collection of all line segment paths in $X$ parallel to the real axis and let $\G$ be the collection of all line segment paths in $X$ parallel to the imaginary axis. Set $\varphi(z):=z$ ($z\in X$). Then $\varphi:X\to X$ is continuously differentiable on $X$ and $\varphi(X)=X$. Clearly $\varphi\circ\gamma\in\F$ for all $\gamma\in\F$. It is not hard to show that, if $\gamma\in\mgen\F,$ then $\gamma\adj\in\F\adj$ and, if $\gamma\in\mgen\G,$ then $\gamma\adj\in\G\adj$. Thus $\mgen\F\cap\mgen\G=\emptyset$ and it follows that $\varphi$ is not $\F$-$\G$-compatible.
\end{example}

Let $X,Y$ be semi-rectifiable compact plane sets, let $\F$ be an effective collection of paths in $X$, and let $\G$ be an effective collection of paths in $Y$. If $\od Y$ is dense in $\fd[\G]1Y$ then, by Lemma \ref{path generated dense set}, the collection $\mgen\G$ is the collection of all admissible paths in $Y$. So, by the comments following Definition \ref{compatibility definition}, any function $\varphi\in\fd1X$ such that $\varphi(X)\subseteq Y$ is automatically $\F$-$\G$-compatible.

\section{Composition of $\F$-differentiable functions}
We now discuss an analogue of the chain rule for $\F$-differentiable functions. We require the following lemma, which is an $\F$-differentiability version of the usual change of variable formula.

\begin{lemma}\label{path composition integral}
Let $X$ be a semi-rectifiable$,$ compact plane set and let $\F$ be an effective collection of paths in $X$. Let $\varphi\in\fd1X$ and let $\gamma:[a,b]\to X\in\F$. Then$,$ for each $f\in C(\varphi(\gamma\adj)),$ we have
\begin{equation}\label{path composition integral formula}
\int_\gamma(f\circ\varphi)\varphi\fdiff 1(z)\;\dm z=\int_{\varphi\circ\gamma}f(z)\;\dm z.
\end{equation}
\end{lemma}
\begin{proof}
By Lemma \ref{composition rectifiable}, $\sigma:=\varphi\circ\gamma$ is a rectifiable path so that the integral on the right hand side of \eqref{path composition integral formula} exists. Fix $f\in C(\sigma\adj)$ and let $\varepsilon>0$. Set $M:=\abs{\varphi\fdiff 1}_X\pl\gamma$ and let $h:=f\circ\varphi:[a,b]\to\C$. Since $h$ is uniformly continuous, there exists $\delta>0$ such that, for each $s,t\in[a,b]$ with $\abs{s-t}<\delta$, we have $\abs{h(t)-h(s)}<\varepsilon/(2M)$. Choose a partition $\mathcal P_0=\{t_0\diff 0,\dotsc,t_m\diff 0\}$ of $[a,b]$ such that $\max_{0\leq j\leq n-1}\abs{t_{j+1}-t_j}<\delta$. For any partition $\mathcal P=\{t_0,\dotsc,t_n\}$ finer than $\mathcal P_0$ and, for each $j\in\{0,\dotsc,n-1\}$, let $\gamma_j^{(\mathcal P)}$ be the restriction of $\gamma$ to $[t_j,t_{j+1}]$. We have
\begin{equation}\label{path composition integral equation 1}
T(\mathcal P):=\left\lvert\sum_{j=0}^{n-1}\int_{\gamma_j^{(\mathcal P)}}(f(\varphi(z))-h(s_j))\varphi\fdiff 1(z)\;\dm z\right\rvert<\frac\varepsilon{2M}\abs{\varphi\fdiff 1}_X\pl\gamma=\frac\varepsilon2,
\end{equation}
for any $s_j\in[t_j,t_{j+1}]$ ($j=0,1,\dotsc,n-1$).

Now fix a partition $\mathcal P=\{t_0,\dotsc,t_n\}$ of $[a,b]$ finer than $\mathcal P_0$ such that, viewing the integral in the RHS of \eqref{path composition integral formula} as a Riemann-Stieltjes integral on $[a,b]$, we have
\begin{equation}\label{path composition integral equation 2}
\left\lvert\sum_{j=0}^{n-1}h(s_j)(\sigma(t_{j+1})-\sigma(t_j))-\int_\sigma f(z)\;\dm z\right\rvert<\frac\varepsilon2,
\end{equation}
for any choice of $s_j\in[t_j,t_{j+1}]$ ($j=0,1,\dotsc,n-1$).

We now {\em claim} that, for this partition $\mathcal P$, we have
\begin{equation}\label{path composition integral equation 3}
\left\lvert\int_{\gamma}f(\varphi(z))\varphi\fdiff 1(z)\;\dm z-\sum_{j=0}^{n-1}h(s_j)(\sigma(t_{j+1})-\sigma(t_j))\right\rvert<\frac\varepsilon2
\end{equation}
for any choice of $s_j\in[t_j,t_{j+1}]$ ($j=0,1,\dotsc,n-1$).

For the remainder of the proof, for each $j\in\{0,\dotsc,n-1\}$, fix $s_j\in[t_j,t_{j+1}]$ and let $S:=\sum_{j=0}^{n-1}h(s_j)(\sigma(t_{j+1})-\sigma(t_j))$.

By the definition of $\varphi\fdiff 1$, we have
\[
\sum_{j=0}^{m-1}h(s_j)\int_{\gamma_j^{(\mathcal P)}}\varphi\fdiff 1(z)\;\dm z=S.
\]
We also have
\[
\left\lvert\int_\gamma f(\varphi(z))\varphi\fdiff 1(z)\;\dm z-\sum_{j=0}^{n-1}h(s_j)\int_{\gamma_j^{(\mathcal P)}}\varphi\fdiff 1(z)\;\dm z\right\rvert
=T(\mathcal P),
\]
where $T(\mathcal P)<\varepsilon/2$ as in \eqref{path composition integral equation 1}.

But now, by \eqref{path composition integral equation 1} and \eqref{path composition integral equation 2}, we have
\[
\left\lvert\int_\gamma f(\varphi(z))\varphi\fdiff 1(z)\;\dm z-\int_\sigma f(w)\;\dm w\right\rvert<\frac\varepsilon2+\frac\varepsilon2=\varepsilon.
\]
This holds for all $\varepsilon>0$ and any choice of the $s_j$, so the result follows.
\end{proof}

We can now state and prove a version of the chain rule for $\F$-differentiable functions.

\begin{theorem}\label{new F chain rule}
Let $X,Y$ be semi-rectifiable$,$ compact plane sets$,$ let $\F$ be an effective collection of paths in $X,$ and let $\G$ be an effective collection of paths on $Y$. Let $\varphi\in\fd1X$ with $\varphi(X)\subseteq Y$. Suppose that $\varphi$ is $\F$-$\G$-compatible. Then$,$ for all $f\in \fd[\G]1Y,$ $f\circ \varphi$ is $\F$-differentiable and $(f\circ\varphi)\fdiff 1=(f\fdiff 1\circ\varphi)\varphi\fdiff 1$.
\end{theorem}
\begin{proof}
Fix $f\in\fd[\G]1Y$ and $\gamma\in\F$. Then, by Lemma \ref{path composition integral}, we have
\[
\int_{\gamma}(f\fdiff 1\circ\varphi)(z)\varphi\fdiff 1(z)\;\dm z=\int_{\varphi\circ\gamma}f\fdiff 1(z)\;\dm z.
\]
Since $\varphi$ is $\F$-$\G$-compatible, we have $\varphi\circ\gamma\in\mgen{\G}$ and so
\[
\int_{\varphi\circ\gamma}f\fdiff 1(z)\;\dm z=f((\varphi\circ\gamma)\pfi)-f((\varphi\circ\gamma)\pin).
\]
But $(\varphi\circ\gamma)\pfi=\varphi(\gamma\pfi)$ and $(\varphi\circ\gamma)\pin=\varphi(\gamma\pin)$. Thus $f\circ\varphi$ is $\F$-differentiable and has $\F$-derivative $(f\fdiff 1\circ\varphi)\varphi\fdiff 1$. This completes the proof.
\end{proof}

As a corollary we obtain the quotient rule for $\F$-differentiable functions. This was originally proved by means of repeated bisection in \cite{chaobankoh2012}.

\begin{corollary}\label{F deriv quotient rule}
Let $X$ be a semi-rectifiable compact plane set$,$ let $\F$ be an effective collection of paths in $X,$ and let $f,g\in\fd 1X$ such that $0\notin g(X)$. Then we have $f/g\in\fd 1X$ and $(f/g)\fdiff 1=(gf\fdiff 1-fg\fdiff 1)/g^2$.
\end{corollary}
\begin{proof}
We first show that $h:=1/g\in\fd 1X$ and that $h\fdiff 1=-g\fdiff 1/g^2$. Since we have $0\notin g(X)$, the function $\varphi(z):=1/z$ ($z\in g(X)$) is continuous and complex-differentiable on $g(X)$, i.e. $\varphi\in\od[1]{g(X)}$. Let $\G$ be the collection of all admissible paths in $g(X)$. Then we have $\varphi\in\od[1]{g(X)}\subseteq\fd[\G]1{g(X)}$, and $g$ is $\F$-$\G$-compatible by the comments following Definition \ref{compatibility definition}. By Theorem \ref{new F chain rule}, $\varphi\circ g\in\fd 1X$ with $(\varphi\circ g)\fdiff 1=(\varphi\fdiff 1\circ g)g\fdiff 1$. However, $\varphi\fdiff 1$ is just the ordinary complex derivative of $\varphi$, and so $\varphi\fdiff 1\circ g=-1/g^2$. Thus $h\fdiff 1=-g\fdiff 1/g^2$. The result now follows from the product rule for $\F$-derivatives, Proposition \ref{F derive old properties}(c).
\end{proof}

By combining Theorem \ref{new F chain rule} with our comments at the end of Section \ref{maximal compatibility section}, we obtain the following corollary.

\begin{corollary}\label{old F chain rule}
Let $X$ be a semi-rectifiable$,$ compact plane set and let $\F$ be an effective collection of paths in $X$. Let $f,g\in\fd 1X$ such that $g(X)\subseteq X$. Suppose that $\od X$ is dense in $\fd 1X$. Then $f\circ g\in\fd1X$ and $(f\fdiff 1\circ g)g\fdiff 1$ is the $\F$-derivative of $f\circ g$.
\end{corollary}

This last result was proved in the first author's PhD thesis (\cite{chaobankoh2012}) using the quotient rule, under the apparently stronger condition that the set of rational functions with no poles on $X$ be dense in $\fd 1X$. See the final section of this paper for an open problem related to this.

\smallskip

By applying Theorem \ref{new F chain rule} inductively, we obtain the Fa\'a di Bruno formula for the composition of $n$-times $\F$-differentiable functions.

\begin{corollary}[Fa\'a di Bruno formula]\label{F Faa di Bruno}
Let $X,Y$ be semi-rectifiable$,$ compact plane sets$,$ let $\F$ be an effective collection of paths in $X,$ and let $\G$ be an effective collection of paths in $Y$. Let $n\in\N$, and let $\varphi\in\fd nX$ with $\varphi(X)\subseteq Y$. Suppose that $\varphi$ is $\F$-$\G$-compatible. Then$,$ for all $f\in\fd[\G] nY,$ $f\circ\varphi\in\fd nX$ and$,$ for each $k\in\{1,2,\dotsc,n\},$ we have
\[
(f\circ\varphi)\fdiff k=\sum_{i=0}^k(f\fdiff i\circ\varphi)\sum\frac{k!}{a_1!\dotsm a_k!}\prod_{j=1}^k\left(\frac{\varphi\fdiff j}{j!}\right)^{a_j},
\]
where the inner sum is over all $a_1,\dotsc,a_k\in\No$ such that 
\[
a_1+\dotsb+a_k=i\quad\text{and}\quad a_1+2a_2+\dotsb+ka_k=k.
\]
\end{corollary}

\section{Homomorphisms}
We now discuss some algebras of infinitely $\F$-differentiable functions, analogous to the algebras $\od[]{X,M}$ introduced by Dales and Davie in \cite{dales1973} (see also the introduction of the present paper). In particular, we describe some sufficient conditions under which a function can induce a homomorphism between these algebras. These conditions are similar to those discussed by Feinstein and Kamowitz in \cite{feinstein2000a}. We begin with some definitions from \cite{abtahi2007,dales1973} and \cite{feinstein2000a}.

\begin{definition}
Let $M=(M_n)_{n=0}^\infty$ be a sequence of positive real numbers. We say that $M$ is an {\em algebra sequence} if $M_0=1$ and$,$ for all $j,k\in\No,$ we have
\[
\binom{j+k}{j}\leq\frac{M_{j+k}}{M_jM_k}.
\]
We define $d(M):=\lim_{n\to\infty}(n!/M_n)^{1/n}$ and we say that $M$ is {\em non-analytic} if $d(M)=0$.
\end{definition}

Let $X$ be a perfect compact plane set and let $M=(M_n)_{n=0}^\infty$ be an algebra sequence. Then the set of all rational functions with no poles on $X$ is contained in $\od[]{X,M}$ if and only if $M$ is non-analytic.

We now discuss the algebras $\ddf{X,M}$ as introduced in \cite{bland2005}.

\begin{definition}
Let $X$ be a semi-rectifiable$,$ compact plane set and let $\F$ be an effective collection of paths in $X$. Let $M=(M_n)_{n=0}^\infty$ be an algebra sequence. We define the normed algebra
\[
\ddf{X,M}:=\left\{f\in\fd{\infty}{X}:\sum_{j=0}^\infty\frac{\abs{f\fdiff j}_X}{M_j}<\infty\right\}
\]
with pointwise operations and the norm
\[
\norm{f}:=\sum_{j=0}^\infty\frac{\abs{f\fdiff j}_X}{M_j}\qquad(f\in\ddf{X,M}).
\]
\end{definition}

The proof that the $\ddf{X,M}$ are indeed algebras is similar to the proof of Theorem~1.6 of \cite{dales1973}. In fact, since $\fd1X$ is complete with the conditions above, it follows that $\ddf{X,M}$ is a Banach function algebra; this is noted in \cite{bland2005}.

Unfortunately, it is not known in general whether the Banach function algebras $\fd{n}{X}$ and $\ddf{X,M}$ are natural on $X$, although some sufficient conditions are given in \cite{chaobankoh2012}. We note that a necessary condition for $\ddf{X,M}$ to be natural is that $M=(M_n)_{n=0}^\infty$ be a non-analytic algebra sequence.

\begin{definition}
Let $X$ be a semi-rectifiable$,$ compact plane set$,$ and let $\F$ be an effective collection of paths in $X$. Let $f\in\fd{\infty}{X}$. We say that $f$ is {\em $\F$-analytic} if
\[
\limsup_{k\to\infty}\left(\frac{\abs{\varphi\fdiff k}_X}{k!}\right)^{1/k}<\infty.
\]
\end{definition}

Note that a function $f\in\fd{\infty}{X}$ which is $\F$-analytic need not be analytic (in the sense of extending to be analytic on a neighbourhood of $X$). For example, let $X$ and $\F$ be as in Example \ref{incompatibility example}, let $M=(M_n)_{n=0}^\infty$ be an algebra sequence, and let $f\in\ddf{X,M}$ with $f\notin \od[]{X,M}$ such that $f$ is $\F$-analytic. Then $f$ is not analytic. For example, we may take take $f(z)={\rm Im}(z)$ here, so that $f\fdiff 1$ is identically $0$.

We now give the main result of this section. No detailed proof is required, since once the Fa\'a di Bruno formula is established the calculations are identical to those from \cite{kamowitz1998} and \cite{feinstein2004a} (see also \cite{feinstein2000a}). Note that we do not assume the naturality of the algebras here.

\begin{theorem}\label{main hom result}
Let $X, Y$ be semi-rectifiable$,$ compact plane sets and let $n\in\N$. Let $\varphi\in\fd{\infty}{X}$ such that $\varphi(X)\subseteq Y$. Suppose that $\varphi$ is $\F$-analytic. Let $\F$ be an effective collection of paths on $X$ and let $\G$ be an effective collection of paths on $Y$ such that $\varphi$ is $\F$-$\G$-compatible.  Let $M=(M_n)_{n=0}^\infty$ be a non-analytic algebra sequence.
\begin{enumerate}
  \item If $\abs{\varphi\fdiff1}_X<1$ then $\varphi$ induces a homomorphism from the algebra $\ddf[\G]{Y,M}$ into $\ddf{X,M}$.
  \item If the sequence $(n^2M_{n-1}/M_{n})$ is bounded and $\abs{\varphi\fdiff1}_X\leq 1$ then $\varphi$ induces a homomorphism from $\ddf[\G]{Y,M}$ into $\ddf{X,M}$.
\end{enumerate}
\end{theorem}

Note that, in the above, $\ddf{X,M}$ and $\ddf[\G]{Y,M}$ are always Banach function algebras, so we do not need to make any additional completeness assumptions.

\smallskip

When $(n^2M_{n-1}/M_n)$ is unbounded, the condition that $\abs{f\fdiff 1}_X\leq 1$ may no longer be sufficient for $\varphi$ to induce a homomorphism from $\ddf[\G]{Y,M}$ into $\ddf{X,M}$. The following example is from \cite{feinstein2004k}.

\begin{example}
Let $I=[0,1]$, let $\F$ be the collection of all admissible paths in $I$, and, for each $n\in\N$, let $M_n=(n!)^{3/2}$. Then $M=(M_n)_{n=0}^\infty$ is a non-analytic algebra sequence such that $n^2M_{n-1}/M_n\to\infty$ as $n\to\infty$. Moreover, since $I$ is uniformly regular, we have $\ddf{I,M}=\od{I,M}$, and $\ddf{I,M}$ is natural on $I$ by \cite[Theorem~4.4.16]{dales2000}. Let $\varphi(t):=(1+t^2)/2$ ($t\in I$).
Then $\varphi$ is an $\F$-analytic map from $I$ into $I$, $\abs{\varphi'}_X\leq 1$, and $\varphi$ is $\F$-$\F$-compatible. However, by \cite[Theorem~3.2]{feinstein2000a}, $\varphi$ does not induce a homomorphism from $\ddf{I,M}$ into $\ddf{I,M}$.
\end{example}

\section{Open problems}
We conclude with some open problems related to the content of this paper.
\begin{question}
Can the assumption that $\varphi$ be $\F$-analytic in Theorem $\ref{main hom result}$ be weakened or removed altogether$?$
\end{question}

We next ask two questions about maximal collections of paths.

\begin{question}
Let $X$ be a semi-rectifiable compact plane set, and let $\F$ be the collection of all Jordan paths in $X$. Is it necessarily true that $\mgen{\F}$ is the collection of all admissible paths in $X?$
\end{question}

\begin{question}
Let $\Gamma$ be an admissible path in $\C$, and set $X=\Gamma\adj$. Let $\F$ be the collection of all subpaths of $\Gamma$. Does $\mgen{\F}$ necessarily include all Jordan paths in $X?$
\end{question}

Let $X$ be a semi-rectifiable compact plane set.
Our final questions concern the density of rational functions and differentiable functions in the algebras $\od X$ and $\fd1X$, respectively. Some partial results were obtained in \cite{bland2005} and \cite{dalefein2010}.

\begin{question}
Is the set of all rational functions with no poles on $X$ always dense in $\od X?$
\end{question}

\begin{question}
Is $\od X$ always dense in $\fd1X$ when $\F$ is the collection of all admissible paths in $X?$
\end{question}


\providecommand{\bysame}{\leavevmode\hbox to3em{\hrulefill}\thinspace}
\providecommand{\MR}{\relax\ifhmode\unskip\space\fi MR }
\providecommand{\MRhref}[2]{%
  \href{http://www.ams.org/mathscinet-getitem?mr=#1}{#2}
}
\providecommand{\href}[2]{#2}

\end{document}